\RequirePackage[l2tabu, orthodox]{nag}
\documentclass[preprint,  11pt]{amsart}

\usepackage{fullpage}
\usepackage{graphicx}
\usepackage{multirow}
\usepackage[square,numbers]{natbib}
\usepackage{tikz-network}
\usepackage{bbm}
\usepackage{amsfonts}
\usepackage{amssymb}
\usepackage{amsmath}
\usepackage{amsthm}
\usepackage{amsbsy}

\usepackage{amssymb} 
\usepackage{verbatim}
\usepackage{bm} 
\usepackage{paralist} 
 \usepackage{color}
 \usepackage{mathrsfs}
 \usepackage{graphicx}
\usepackage{hyperref}

\def\<{\langle}
\def\>{\rangle}

\newcommand{\ba}{\[\begin{aligned}}
\newcommand{\ea}{\end{aligned}\]}
\newcommand\mnote[1]{} 
\newcommand{\beq}[1]{\begin{equation}\label{#1}}
\newcommand\eeq{\end{equation}}
\newcommand\ben{\begin{equation}}
\newcommand\een{\end{equation}}
\newcommand\bes{\begin{eqnarray*}}
\newcommand\ees{\end{eqnarray*}}
\newcommand\besn{\begin{eqnarray}}
\newcommand\eesn{\end{eqnarray}}

\def\bthm{\begin{theorem}}
\def\ethm{\end{theorem}}
\def\bdefn{\begin{definition}}
\def\edefn{\end{definition}}
\newcommand{\benu}{\begin{enumerate}\setlength\itemsep{6pt}}
\newcommand{\beit}{\begin{itemize}\setlength\itemsep{3pt}}
\def\eenu{\end{enumerate}}
\def\eeit{\end{itemize}}
\def\beds{\begin{description}}
\def\eeds{\end{description}}
\def\bepr{\begin{problem}}
\def\eepr{\end{problem}}

\def\bprf{\begin{proof}}
\def\eprf{\end{proof}}
\def\berk{\begin{remark}}
\def\eerk{\end{remark}}
\def\bex{\begin{exercise}}
\def\eex{\end{exercise}}
\def\beg{\begin{example}}
\def\eeg{\end{example}}

\def\PP{{\mathcal P}}

\newcommand{\sm}{{\raise0.3ex\hbox{$\scriptstyle \setminus$}}}

\def\alp{\alpha}

\renewcommand\phi{\varphi}

\theoremstyle{plain} 
    \newtheorem{theorem}{Theorem}
    \newtheorem{lemma}[theorem]{Lemma}
    \newtheorem{proposition}[theorem]{Proposition}

\theoremstyle{definition} 
    \newtheorem{definition}[theorem]{Definition}

    \newtheorem{exercise}[theorem]{Exercise}
    \newtheorem{problem}[theorem]{Problem}
        \newtheorem{remark}[theorem]{Remark}
    \newtheorem{example}[theorem]{Example}

\renewcommand{\bex}{\indent\begin{exercise}}

\usepackage{palatino}

\hypersetup{
    colorlinks=true, 
    linktoc=all,     
    linkcolor=blue,  
}
\begin{document}

\title{On Hadamard powers of Random Wishart matrices}
\author[J.~S.~Baslingker]{Jnaneshwar Baslingker}
\address{Department of Mathematics, Indian Institute of Science, Bangalore-560012, India.}
\email{jnaneshwarb@iisc.ac.in}
\date{}
\begin{abstract}
A famous result of Horn and Fitzgerald is that the $\beta$-th Hadamard power of any $n\times n$ positive semi-definite (p.s.d) matrix with non-negative entries is p.s.d $\forall \beta\geq n-2$ and is not necessarliy p.s.d for $\beta< n-2,$ with $\ \beta\notin \mathbb{N}$. In this article, we study this question for random Wishart matrix $A_n:={X_nX_n^T}$, where $X_n$ is $n\times n$ matrix with i.i.d. Gaussians. It is shown that applying $x\rightarrow |x|^{\alpha}$ entrywise to $A_n$, the resulting matrix is p.s.d, with high probability, for $\alpha>1$ and is not p.s.d, with high probability, for $\alpha<1$. It is also shown that if $X_n$ are $\lfloor n^{s}\rfloor\times n$ matrices, for any $s<1$, the transition of positivity occurs at the exponent $\alp=s$.
\end{abstract}
\maketitle

\section{Introduction}

Entrywise exponents of matrices preserving positive semi-definiteness has been a topic of active research.  An important theorem in this field is the result of Horn and Fitzgerald \cite{FH77}. Let $\PP_n^+$ denote the set of $n\times n$ p.s.d. matrices with non-negative entries. Schur product theorem gives us that the $m$-th Hadamard power $A^{\circ m}:=[a_{ij}^m]$ of any p.s.d. matrix $A=[a_{ij}]\in \PP_n^+$  is again p.s.d. for every positive integer $m$. Horn and Fitzgerald proved that $n-2$ is the `critical exponent' for such matrices, i.e., $n-2$ is the least number for which $A^{\circ \alpha}\in \PP_n^+$ for every $A\in \PP_n^+$ and for every real number $\alpha \ge n-2$. They considered the matrix $A\in \PP_n^+$ with $(i,j)$-th entry $1+\varepsilon ij$ and showed that if $\alpha$ is not an integer and $0<\alpha<n-2$, then $A^{\circ \alpha}$ is not positive semi-definite for a sufficiently small positive number $\varepsilon$ (see \cite{AK}).

We consider a random matrix version of this problem. Let $X:=\left[X_{ij}\right]$ be a $n \times n$ matrix, where $X_{ij}$ are i.i.d standard normal random variables. 
Define  $A_{n}:=\frac{XX^T}{n}$ and $|A_{n}|^{\circ\alpha}$ as the matrix obtained by applying $x\rightarrow |x|^{\alpha}$ function entrywise to $A_{n}$. Let $B_{n,\alpha}:=|A_{n}|^{\circ\alpha}$. 
%

We are interested in the values of real $\alpha>0$ for which the matrix $B_{n,\alpha}$ is positive semi-definite, with high probability. Simulations show that  for large values of $n$, if $\alpha>1$ then with high probability, $B_{n,\alpha}$ is positive semi-definite and for $\alpha<1$, with high probability, $B_{n,\alpha}$ is not positive semi-definite (as shown in Table \ref{Table 1}).

We state and prove the theorem that these observations from simulations are indeed true. In fact we prove a stronger result. Fix any $s\leq 1$ and let $m= \lfloor n^{s}\rfloor$. Let $X_n:=\left[X_{ij}\right]$ be a $m \times n$ matrix, where $X_{ij}$ are i.i.d standard normal random variables. Define  $A_{n,s}:=\frac{X_nX_n^T}{n}$ and $B_{n,\alpha,s}:=|A_{n,s}|^{\circ\alpha}$. Let $\lambda_1(A)$ denote the smallest eigenvalue of $A$. We prove the following main result.
\begin{theorem}\label{main theorem}
 $\exists \varepsilon_s>0$ such that for $\alpha>s$, as $n\rightarrow \infty$
\begin{align*}
\mathbb{P}\left(\lambda_{1}(B_{n,\alpha,s})\geq \varepsilon_s\right)&\rightarrow 1.   
\end{align*}

For $\alpha<s$, as $n\rightarrow \infty$
\begin{align*}
\mathbb{P}\left(\lambda_{1}(B_{n,\alpha,s})<0\right)&\rightarrow 1.   
\end{align*}
\end{theorem} 

\begin{remark}
Simulations show Theorem \ref{main theorem} holds if i.i.d Gaussians are replaced by other i.i.d random variables with finite second moment like Uniform$(0,1)$, Exp($1$) and even heavy tailed distributions like Cauchy distribution, distributions with densities $f(x)=bx^{-1-b},\forall x\geq 1$, all with transition of positivity at exponent $\alp=s$. Note that in the last case one does not have finite mean if $b$ is small. This suggests that the transition of matrix positivity happens for a large family of distributions. In this direction we prove the below proposition where we show that $B_{n,\alp,s}$ is p.s.d for the range of $\alpha>2s$,  when $X_n$ has sub-Gaussian entries. 
\end{remark}

\begin{proposition}\label{alpha>2}
Let the entries of $X_n$ be i.i.d sub-Gaussian random variables with mean $0$ and unit variance. Fix $\alpha>2s$ and $\varepsilon>0$. Define $B_{n,\alp,s}$ as before. Then as $n\rightarrow \infty$
\begin{align}
\mathbb{P}\left(\lambda_{1}(B_{n,\alp,s})\leq 1-\varepsilon\right)&\rightarrow 0,   \\
\mathbb{P}\left(\lambda_{m}(B_{n,\alp,s})\geq 1+\varepsilon\right)&\rightarrow 0.   
\end{align}
\end{proposition} 

\begin{table}[h!]
\centering
\begin{tabular}{ |c|c|c|c| } 

\hline
$s$ & $\alpha$ & $\lambda_1$  \\
\hline
$1$ & $0.98$ & $-0.288$ \\ 
$1$ & $0.99$ & $-0.246$ \\ 
$1$ & $1.06$ & $0.016$  \\ 
$1$ & $1.07$ & $0.046$ \\ 
$0.8$ & $0.78$ & $-0.076$ \\ 
$0.8$ & $0.79$ & $-0.049$ \\ 
$0.8$ & $0.81$ & $0.017$  \\ 
$0.8$ & $0.82$ & $0.041$ \\ 
\hline

\end{tabular}
\caption{Table of smallest eigenvalues for varying $\alpha$ and $s$ with $n=5000$.\label{Table 1}}
\end{table}
\begin{remark}
Although Theorem \ref{main theorem} and Proposition \ref{alpha>2} hold for $m=\Theta(n^s)$, for definiteness we fix $m=\lfloor n^s\rfloor$. For $m=a\times n$ for fixed $a>0$, the transition of positivity is at exponent $1$. For the critical exponent to be less than $1$, we need $m=\Theta(n^s)$ with $ s<1$, which is much smaller, unlike in the study of spectrum of Wishart matrices.
\end{remark}
A standard way to study the distribution of eigenvalues of a random matrix is to look at the limit of empirical spectral distributions using method of moments. For example, Wigner's proof of semi-circle law for Gaussian ensemble uses this method (For more see \cite{AGZ}). In our case, the entries of the matrix $B_{n,\alp,s}$ are sums of products of random variables and the entries on the same row or column are correlated. The entrywise absolute fractional power makes this problem intractable, if we try to use method of moments. As we are interested
only in the existence of negative eigenvalues, we manage to avoid computing all the moments.

\subsection{Outline of the paper:}
\hfill

First we prove Proposition \ref{alpha>2} in Section \ref{sec a>2}. This is done using Gershgorin's circle theorem and the sub-exponential Bernstein's inequality. Note that this proposition is not needed to prove Theorem \ref{main theorem}.

The proof of Theorem \ref{main theorem} is divided into two parts. In the first part of the proof, we consider the range $\alp<s$.  Let $C_{n,\alp,s}:=\frac{B_{n,\alp,s}}{n^{\frac{s-\alpha}{2}}}$. For ease of notation, we write $C_{n,\alp,s}$ as $C_m$. $C_m$ is a $m\times m$ matrix where $m=\lfloor n^s\rfloor$.
  Define the diagonal matrix $D_m$, with $D_m(i,i):=C_m(i,i)-\frac{\ell_{\alpha}}{n^{s/2}}$ and $E_m:=C_m-D_m-\frac{\ell_{\alpha}}{n^{s/2}}J_m$, where $\ell_{\alpha}$ is as defined in Subsection \ref{Notations}. We use the following lemma, whose proof is given in Section \ref{sec a<1}, to conclude that EESD of $B_{n,\alp,s}$ has positive weight on negative reals. 
\begin{lemma}\label{Lemma 1 of alpha<1}
Let $\bar{\mu}_{E_m}$ be the EESD of $E_m$. Then\\
i) Limit of first moment of $\bar{\mu}_{E_m}$ is $0$\\
ii) Limit of second moment of $\bar{\mu}_{E_m}$ is a positive constant\\
iii) The fourth moments of $\bar{\mu}_{E_m}$ are uniformly bounded.
\end{lemma}
  Using a concentration of measure result, we show that with high probability, $B_{n,\alp,s}$ has negative eigenvalues. This is done in Section \ref{sec a<1}.

In the second part of the proof, we consider the range $s<\alp$. We further divide this range by looking at $\left(\frac{k+1}{k}\right)s<\alp$, where $k$ is an integer greater than $1$ and let $k\rightarrow\infty$. For $\left(\frac{k+1}{k}\right)s<\alp$, we consider $C_m$, a modification of $B_{n,\alp,s}$, whose EESD has $2k$-th moment converging to $0$ to conclude that the probability of $B_{n,\alp,s}$ having a negative eigenvalue converges to 0. We then let $k$ be arbitrarily large. This is done in Section \ref{sec a>1}.



\subsection{Notation}\label{Notations}

\hfill\\

1)   $m=\lfloor n^s\rfloor$.

2)   $\lambda_1(A)$ and $\lambda_m(A)$ denote the smallest and largest eigenvalues of $A$ respectively.

3)   $R_i$ denotes the $i$th row of $X_n$. ($R_i^T\sim N(0, I_n)$ in Sections \ref{sec a<1}, \ref{sec a>1} but not necessarily in Section \ref{sec a>2}).

4)   $\rho_{ij}=\frac{\langle R_i,R_j\rangle}{\lVert R_i\rVert\lVert R_j\rVert}$.

5)  $\ell_{\alpha}=\mathbb{E}[|Z|^{\alpha}]$, where $Z$ is a standard normal random variable.

6)   $J_n=$ All ones matrix of size $n\times n$ and $I_n=$ $n\times n$ identity matrix.

7)   $\mathcal{F}_{i,j}=$ The sigma algebra generated from the $i$th row and $j$th row of $X_n$.

8)   $\sigma_i=\lVert R_i\rVert/\sqrt{n}$.

9)   $Y_{ij}=\mathbb{E}\left[\left(\left|\frac{\langle R_i,R_k\rangle}{\sqrt{n}}\right|^{\alpha}-m_{\alpha}\right)\left(\left|\frac{\langle R_k,R_j\rangle}{\sqrt{n}}\right|^{\alpha}-m_{\alpha}\right)\bigg|\mathcal{F}_{i,j}\right]$.

\section{Proof of Proposition \ref{alpha>2}} \label{sec a>2}
In this section we prove Proposition \ref{alpha>2}.
\begin{proof}[Proof of Proposition~\ref{alpha>2}] For ease of notation, we write $B_{n,\alp,s}$ as $B_n$.
The diagonal entries of $B_n$ are of the form $\left(\frac{\langle R_i,R_i\rangle}{n}\right)^{\alpha}$ and off-diagonal entries are of the form $\left|\frac{\langle R_i,R_j\rangle}{n}\right|^{\alpha}$. Note that all the off-diagonal entries are identically distributed and all the diagonal entries are identically distributed.
First we give an upper bound for the probability that $\sum_{i=2}^{m}(B_n)_{1i}>\varepsilon$. 
\begin{align*}
\mathbb{P}\left(\sum_{i=2}^{m}(B_n)_{1i}>\varepsilon\right)\leq m \mathbb{P}\left((B_n)_{12}>\frac{\varepsilon}{m}\right).
\end{align*}  

Note that $(B_n)_{12}$ is a function of sum of $n$ independent sub-exponential random variables (product of independent Gaussians is sub-exponential (Lemma $2.7.7$ of \cite{RV}). We now recall the Bernstein inequality for sub-exponential random variables from \cite{RV}.

\begin{theorem}{(Theorem $2.8.1$ of \mbox{\cite{RV})}}\label{Bernstein}
Let $X_1,X_2,\dots,X_N$ be independent, mean zero, sub-exponential random variables. Then, for every $t\geq 0$, we have  
\begin{align*}
\mathbb{P}\left({\left|\sum_{i=1}^NX_i\right|\geq t}\right)\leq 2
\exp\left[-c\min{\left(\frac{t^2}{\sum_{i=1}^N\lVert X_i\rVert_{\psi_1}^2},\frac{t}{\max_i{\lVert X_i\rVert}_{\psi_1}}\right)}\right]\end{align*}
where $c>0$ is an absolute constant and $\lVert X\rVert_{\psi_1}$ is the sub-exponential norm of $X$.
\end{theorem}

Bernstein's inequality and the fact that $m=\lfloor n ^s\rfloor$ gives us that
\begin{align*}
\mathbb{P}\left((B_n)_{12}>\frac{\varepsilon}{m}\right)&=\mathbb{P}\left(\left|\langle R_1,R_2\rangle\right|\geq n\left(\frac{\varepsilon}{m}\right)^{1/\alpha}\right)\\
&\leq 2\exp\left(-c_1n^{1-\frac{2s}{\alpha}}\right)
\end{align*} 
for some constant $c_1=c_1(\varepsilon)$. 
This implies that 
\begin{align*}
 \mathbb{P}\left(\sum_{i=2}^{m}(B_n)_{1i}>\varepsilon\right)\leq 2m\exp(-c_1n^{1-\frac{2s}{\alpha}}). 
\end{align*}

Using the identical distribution of off-diagonal entries, we get that 
\begin{align}\label{eq:Gersh_1}
\mathbb{P}\left(\bigcup_{i=1}^{m}\left(\sum_{j=1,j\neq i}^{m}(B_n)_{ij}>\varepsilon\right)\right)\leq 2m^2\exp(-c_1n^{1-\frac{2s}{\alpha}}).
\end{align} 

For the diagonal entry $(B_n)_{11}$, we have 
\begin{align*}
\mathbb{P}\left((B_n)_{11}\leq 1-\varepsilon\right)&=\mathbb{P}\left(\frac{\langle R_1,R_1\rangle}{n}\leq (1-\varepsilon)^{1/\alpha}\right)\\
&\leq \mathbb{P}\left({\langle R_1,R_1\rangle}-n\leq n((1-\varepsilon)^{1/\alpha}-1)\right)\\
&\leq 2\exp\left(-c_2n\right),
\end{align*} 
for a constant $c_2=c_2(\varepsilon,\alpha)$. Here we have used Theorem \ref{Bernstein} in the last inequality, as ${\langle R_1,R_1\rangle}$ $-n$ is a sum of $n$ mean $0$, i.i.d, sub-exponential random variables and $t=n((1-\varepsilon)^{1/\alpha}-1)$.

This implies that
\begin{align}\label{eq:Gersh_2}
\mathbb{P}\left(\bigcup_{i=1}^{m}\left((B_n)_{ii}\leq 1-\varepsilon\right)\right)\leq 2m\exp(-c_2n).
\end{align}
Similarly 
\begin{align}\label{eq:Gersh_3}
\mathbb{P}\left(\bigcup_{i=1}^{m}\left((B_n)_{ii}\geq 1+\varepsilon\right)\right)\leq 2m\exp(-c_2n).
\end{align}

Applying Gershgorin circle theorem (Theorem $6.1.1$ of \cite{HJ}) to $B_n$, using \eqref{eq:Gersh_1}, \eqref{eq:Gersh_2}, \eqref{eq:Gersh_3}, gives us that, with probability at least $1-4m^2\exp(-c_3n^{1-\frac{2s}{\alpha}}),\ \lambda_{1}\geq 1-2\varepsilon$ and $\lambda_{m}\leq 1+2\varepsilon$. Here $c_3>0$ depends on $\varepsilon$ and $\alpha$. As $\alpha>2s$, this completes the proof of Proposition \ref{alpha>2}. 
\end{proof}

\section{$\alpha<s$ range}\label{sec a<1}
In this section we prove Theorem \ref{main theorem} for the range $\alp<s$.  We define a few terms here which will be used in the rest of the article.
Empirical spectral distribution of a symmetric random matrix $A_n$ is the random probability measure $\mu_{A_n}:=\frac{1}{n}\sum\limits_{i=1}^n\delta_{\lambda_i}$, where $\lambda_i$s are the eigenvalues of $A_n$. Expected empirical spectral distribution(EESD) of $A_n$ is the probability measure $\bar{\mu}_{A_n}$ such that $\int_{\mathbb{R}}fd\bar{\mu}_{A_n}=\mathbb{E}\left[\int_{\mathbb{R}}f\ d\mu_{A_n}\right]$, for all bounded continuous functions $f$ (For more see \cite{AGZ}). We prove the following lemma which implies Theorem \ref{main theorem} for the range $\alp<s$.

\label {sec a<1}
\begin{lemma}\label{alpha<1}
Fix $\alpha<s$. Then $\mathbb{P}(\lambda_1(C_{n,\alp,s})<0)\rightarrow 1$, as $n\rightarrow \infty$.
\end{lemma}

\begin{proof}[Proof of Lemma ~\ref{alpha<1}]


 We complete the proof of Lemma \ref{alpha<1} assuming Lemma \ref{Lemma 1 of alpha<1} and then provide the proof of Lemma \ref{Lemma 1 of alpha<1}. For the sake of contradiction assume that $\mathbb{P}\left(\lambda_1(C_m)<0\right)$ does not converge to $1$, then by going to a subsequence we may assume that $\exists \ \varepsilon>0$ such that
$\mathbb{P}\left(\lambda_1(C_m)>0\right)>\varepsilon$ and $\bar{\mu}_{E_m}$ converge weakly to some probability distribution $\mu$ (Using $(ii)$ of Lemma \ref{Lemma 1 of alpha<1} we get the tightness of $\bar{\mu}_{E_m}$).

 Now $\mu$ must have mean $0$, positive variance. Indeed, if a sequence of probability distributions $\bar{\mu}_{E_m}$ converge weakly to $\mu$, then by Skhorokhod's theorem, on some probability space there exist random variables $T_m\sim \bar{\mu}_{E_m}$ and $T\sim \mu$ such that $T_m$ converge almost surely to $T$. Now as $\bar{\mu}_{E_{m}}$ have uniform bound on second moments, we get that $T_m$ are uniformly integrable. This implies that the first moment of $T$ is the limit of first moments of $T_m$. Similarly as the fourth moments of $T_m$ are uniformly bounded, the second moment of $T$ is the limit of second moments of $T_m$. Thus $\mu$ has mean $0$, positive variance.
 
As $\mu$ has zero mean and positive variance, $\mu(-\infty,-\omega)\geq \eta$ for some $\eta,\omega>0$. This gives us that
\begin{align}\label{eq:limit}
    \bar{\mu}_{E_m}(-\infty,-\omega)>\frac{\eta}{2}
\end{align}
for large enough $n$. We would like to say with high probability, empirical spectral distributions of $E_m$ also have positive weight on the negative reals. This would imply the existence of negative eigenvalues, with high probability. Here we make use of the following McDiarmid-type concentration result due to Guntuboyina and Leeb \cite{GL}. For a $n\times n$ symmetric matrix $A$, let $\mu_A$ denote the probability measure $\mu_A:=\frac{1}{n}\sum_{i=1}^n\delta_{\lambda_i}$, where $\lambda_i$s are the eigenvalues of $A$.
 Let $F_{\mu_A}$ denote the cumulative distribution function of ${\mu}_{A}$ and $F_{\mu_A}(f)=\int_{\mathbb{R}}f d\mu_{A}$. The Kolmogorov-Smirnov distance between two probability measures $\mu,\mu'$ is defined as $d_{{KS}}(\mu,\mu'):=\lVert F_{\mu}-F_{\mu'}\rVert_{\infty}$. Let $V_g([a,b])$ denote the total variation of the function $g$ on an interval $[a,b]$ and $V_g(\mathbb{R}):=\sup_{[a,b]}V_g([a,b])$.
\begin{theorem}[Theorem $6$ of \cite{GL}]\label{Guntuboyina}
Let $M$ be a random symmetric $n\times n$ matrix that is a function of $m$ independent random quantities $Y_1,Y_2,\dots,Y_m,$ i.e., $M=M(Y_1,Y_2,\dots,Y_m)$. Write $M_{(i)}$ for the matrix obtained from $M$ after replacing $Y_i$ by an independent copy, i.e., $M_{(i)}=M(Y_1,\dots,Y_{i-1},Y_i^*,Y_{i+1},\dots,Y_m)$ where $Y_{i}^*$ is distributed as $Y_i$ and independent of $Y_1,Y_2\dots,Y_m.$ For $S=M/\sqrt{m}$ and $S_{(i)}=M_{(i)}/\sqrt{m}$, assume that 
\begin{align*}
\lVert F_S-F_{S_{(i)}}\rVert_\infty\leq \frac{r}{n}
\end{align*} 
holds (almost surely) for each $i=1,2,\dots,m$ and for some (fixed) integer $r$. Finally, assume that $g:\mathbb{R}\rightarrow\mathbb{R}$ is of bounded variation on $\mathbb{R}$. For each $\varepsilon>0,$ we then have
\begin{align*}
\mathbb{P}\left(|F_S(g)-\mathbb{E}[F_S(g)]|\geq \varepsilon\right)\leq 2\exp\left[-\frac{n^2 2\varepsilon^2}{mr^2V_g^2(\mathbb{R})}\right].
\end{align*}
\end{theorem} 

We apply Theorem \ref{Guntuboyina} where $E_m$ is the matrix $M$ which is a function of the $\lfloor n^s\rfloor$ rows (independent) of $X_n$. In order to apply Theorem \ref{Guntuboyina}, we need to show  
\begin{align}\label{rank condition}
\lVert F_{E_m}-F_{E_{m{{(i)}}}}\rVert_\infty\leq \frac{r}{\lfloor n^s\rfloor}
\end{align} 
almost surely. Here $E_{m(i)}$ is the matrix obtained when $i$th row of $X_n$ is replaced by an independent and identical copy. To show \eqref{rank condition}, we use that the rank($E_m-E_{m(i)})\leq 2$ and the
 standard rank inequality (Lemma $2.5$ of \cite{CB}) which gives us

\begin{align*}
\lVert F_{E_m}-F_{E_{m{{(i)}}}}\rVert_\infty\leq \frac{2}{\lfloor n^s\rfloor}.
\end{align*}

Note that $V_f(\mathbb{R})$ is finite and independent of $n$. We can now apply Theorem \ref{Guntuboyina} to the matrices $E_m$. Using the function $f=\mathbbm{1}_{(-\infty,-\omega)}$ as the bounded variation function and applying Theorem \ref{Guntuboyina}, we get 
\begin{align}\label{eq:concentration}
\mathbb{P}\left(|F_{E_m}(f)-\mathbb{E}[F_{E_m}(f)]|\geq \eta/4\right)\leq 2\exp\left(-{c\lfloor n^s\rfloor\eta^2}\right)
\end{align}
for some $c>0$.
Using \eqref{eq:limit} and \eqref{eq:concentration}, we get that, for large enough $n$
\begin{align*}
\mathbb{P}(\mbox{fraction of eigenvalues of }E_{n} \mbox{ less than }-\omega/2\geq \eta/4)\geq 1-\frac{\varepsilon}{2}. 
\end{align*}

$E_{m}$ is almost $C_{m}$, with diagonals made $0$ and then off-diagonals are subtracted by $\ell_{\alpha}/\lfloor n^s\rfloor$.

Using \eqref{eq:Gersh_3}, it can be seen that 
\begin{align}
\mathbb{P}\left(\bigcup_{i=1}^{m}\left((C_n)_{ii}\geq n^{\frac{\alpha-s}{2}}(1+\varepsilon)\right)\right)\leq 2m\exp(-c_2n)\label{eq:Gersh_4}\\
\mathbb{P}\left(\bigcup_{i=1}^{m}\left((D_n)_{ii}\geq n^{\frac{\alpha-s}{2}}\left(1+\varepsilon-\frac{\ell_{\alpha}}{n^{\alpha/2}}\right)\right)\right)\leq 2m\exp(-c_2n).\label{eq:Gersh_5}
\end{align}

 Weyl's inequality (Theorem $4.3.1$ of \cite{HJ}) bounds the amount of perturbation of eigenvalues due to perturbation of a matrix. Using Weyl's inequality, along with \eqref{eq:Gersh_5} gives that, 
\begin{align*}
\mathbb{P}(\mbox{fraction of eigenvalues of }E_{m}+D_{m} \mbox{ less than }-\frac{\omega}{2}+n^{\frac{\alpha-s}{2}}(1+\varepsilon-\frac{m_{\alpha}}{n^{\alpha/2}})\geq \eta/4)\\
\geq 1-\frac{\varepsilon}{2}-2m\exp(-c_2n). 
\end{align*}

As rank($E_{m}+D_{m}-C_{m})=1 $ and $\alp<s$, using rank inequality (Lemma $2.5$ of \cite{CB}) again, we get that 

\begin{align*}
\mathbb{P}(\mbox{all the eigenvalues of }C_{m}\mbox{ are non-negative})<\frac{\varepsilon}{2}+2m\exp(-c_2n),
\end{align*}
which contradicts the earlier assumption. This completes the proof of Lemma \ref{alpha<1}.
\end{proof}
We now prove Lemma \ref{Lemma 1 of alpha<1}.

\begin{proof}[Proof of Lemma ~\ref{Lemma 1 of alpha<1}]

\textbf{Computation of moments of $\bar{\mu}_{E_m}$}: Before we start the computations, we make a note of the form of entries of $E_m$.

Diagonal entries: $(E_m)_{ii}=0$

Off diagonal entries: $(E_m)_{ij}=\frac{1}{n^{s/2}}\left(\left|\frac{\langle R_i,R_j\rangle}{\sqrt{n}}\right|^{\alpha}-\ell_{\alpha}\right)$

  We prove limits of first and second moments of $\bar{\mu}_{E_m}$ are $0$ and a positive value. 
  
  \textbf{Limit of first moments}: $\int_{\mathbb{R}}x\ d\bar{\mu}_{E_m}(x)=\mathbb{E}\left[\int_{\mathbb{R}}x\ d\mu_{E_m}(x)\right]=\frac{1}{m}\sum_{i=1}^m\mathbb{E}[(E_m)_{ii}]=0$. Hence
\begin{align*}
  \lim_{n\rightarrow\infty}\int_{\mathbb{R}}x\ d\bar{\mu}_{E_m}(x)&=0
\end{align*}
  
  \textbf{Limit of second moments}: $\int_{\mathbb{R}}x^2\ d\bar{\mu}_{E_m}(x)=\mathbb{E}\left[\int_{\mathbb{R}}x^2\ d\mu_{E_m}(x)\right]=\frac{1}{m}\sum_{i,j}(E_m)_{ij}^2$. As the off-diagonal entries are identically distributed, it is enough to look at the limit of
  $\sum_{i=1}^m\mathbb{E}[((E_m)_{1i})^2]$. 
\begin{align*}
   \lim_{n\rightarrow\infty} \sum_{i=1}^m\mathbb{E}[((E_m)_{1i})^2]= \lim_{n\rightarrow\infty} (m-1) \mathbb{E}[((E_m)_{12})^2] 
\end{align*}  
Using central limit theorem, uniform bound on $\mathbb{E}\left[\left(\frac{\langle R_1,R_2\rangle}{\sqrt{n}}\right)^4\right]$ and $m=\lfloor n^s\rfloor$, it is easy to see that the limit is $\mathbb{E}\left[\left(|Z|^{\alpha}-\ell_{\alpha}\right)^2\right]$.
  
  We now prove that the fourth moments of $\bar{\mu}_{E_m}$ are uniformly bounded. 
  
  \textbf{Uniform bound of fourth moments}:
  \begin{align*}
  \int_{\mathbb{R}}x^4\ d\bar{\mu}_{E_m}(x)=\frac{1}{m}\sum_{i_1i_2i_3i_4}\mathbb{E}\left[(E_m)_{i_1i_2}(E_m)_{i_2i_3}(E_m)_{i_3i_4}(E_m)_{i_4i_1}\right]. 
  \end{align*}
  This is a sum of expectations with each term corresponding to a closed walk of length $4$ on the complete graph $K_m$. It is enough to look at closed walks starting and ending at vertex $1$. Such walks can visit $2,3$ or $4$ different vertices, including the vertex $1$. 
\begin{align*}
  \int_{\mathbb{R}}x^4\ d\bar{\mu}_{E_m}(x)=\sum_{i,j,k\neq 1}\mathbb{E}\left[(E_m)_{1i}^4\right]+\mathbb{E}\left[(E_m)_{1i}^2(E_m)_{1j}^2\right]\\+\mathbb{E}\left[(E_m)_{1i}^2(E_m)_{ij}^2\right]+\mathbb{E}\left[(E_m)_{1i}(E_m)_{ij}(E_m)_{jk}(E_m)_{k1}\right]
\end{align*}
    
    The four terms in the above equation correspond to four different types of walks as shown below.

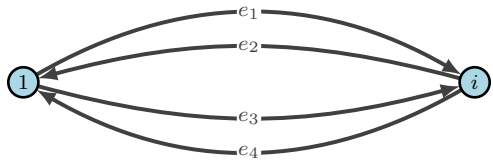
\begin{figure}[h!]
\begin{tikzpicture}
 \Vertex[label=$1$, size=0.4, x=10,y=10]{A} \Vertex[label=$i$, size=0.4, x=16,y=10]{B}
  \Edge[label=$e_1$,bend=30,Direct](A)(B)
  \Edge[label=$e_2$,bend=-15,Direct](B)(A)
  \Edge[label=$e_3$,bend=-15,Direct](A)(B)
  \Edge[label=$e_4$,bend=30,Direct](B)(A)
\end{tikzpicture}
\caption{The walk corresponding to $\mathbb{E}\left[(E_m)_{1i}^4\right]$}
\end{figure}

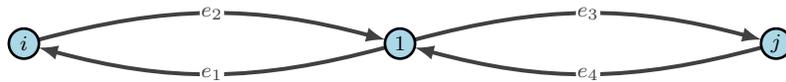
\begin{figure}[h!]
\begin{tikzpicture}
 \Vertex[label=$i$, size=0.4, x=5,y=10]{i} \Vertex[label=$1$, size=0.4, x=10,y=10]{1} \Vertex[label=$j$, size=0.4, x=15,y=10]{j}
  \Edge[label=$e_1$,bend=15,Direct](1)(i)
  \Edge[label=$e_2$,bend=15,Direct](i)(1)
  \Edge[label=$e_3$,bend=15,Direct](1)(j)
  \Edge[label=$e_4$,bend=15,Direct](j)(1)
\end{tikzpicture}
\caption{The walk corresponding to $\mathbb{E}\left[(E_m)_{1i}^2(E_m)_{1j}^2\right]$}
\end{figure}

\begin{figure}[h!]
\begin{tikzpicture}
 \Vertex[label=$1$, size=0.4, x=5,y=10]{1} \Vertex[label=$i$, size=0.4, x=10,y=10]{i} \Vertex[label=$j$, size=0.4, x=15,y=10]{j}
  \Edge[label=$e_1$,bend=15,Direct](1)(i)
  \Edge[label=$e_2$,bend=15,Direct](i)(j)
  \Edge[label=$e_3$,bend=15,Direct](j)(i)
  \Edge[label=$e_4$,bend=15,Direct](i)(1)
\end{tikzpicture}
\caption{The walk corresponding to $\mathbb{E}\left[(E_m)_{1i}^2(E_m)_{ij}^2\right]$}
\end{figure}
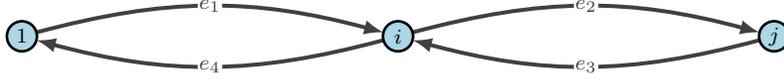

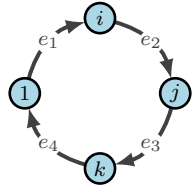
\begin{figure}[h!]
\begin{tikzpicture}
 \Vertex[label=$1$, size=0.4, x=5,y=10]{1} \Vertex[label=$i$, size=0.4, x=6,y=11]{i} \Vertex[label=$j$, size=0.4, x=7,y=10]{j}
 \Vertex[label=$k$, size=0.4, x=6,y=9]{k}
  \Edge[label=$e_1$,bend=30,Direct](1)(i)
  \Edge[label=$e_2$,bend=30,Direct](i)(j)
  \Edge[label=$e_3$,bend=30,Direct](j)(k)
  \Edge[label=$e_4$,bend=30,Direct](k)(1)
\end{tikzpicture}
\caption{The walk corresponding to $\mathbb{E}\left[(E_m)_{1i}(E_m)_{ij}(E_m)_{jk}(E_m)_{k1}\right]$}
\end{figure}

Using the fact that off-diagonal entries of $E_m$ are identically distributed, uniform bound on $\mathbb{E}\left[\left(\frac{\langle R_1,R_2\rangle}{\sqrt{n}}\right)^4\right]$ and central limit theorem, it can be seen that
\begin{align}\label{eq:2V}
\lim_{n\rightarrow\infty}\sum_{i\neq 1}\mathbb{E}\left[(E_m)_{1i}^4\right]=\lim_{n\rightarrow\infty}\frac{(m-1)}{m^2}\mathbb{E}\left[\left(\left|\frac{\langle R_1,R_2\rangle}{\sqrt{n}}\right|^{\alpha}-\ell_{\alpha}\right)^4\right]=0.
\end{align}
Using a similar argument as above it can be seen that
\begin{align}\label{eq:3V}
\lim_{n\rightarrow\infty}\sum_{i,j\neq 1}\mathbb{E}\left[(E_m)_{1i}^2(E_m)_{1j}^2\right]= \lim_{n\rightarrow\infty}\sum_{i,j\neq 1}\mathbb{E}\left[(E_m)_{1i}^2(E_m)_{ij}^2\right]=\mathbb{E}\left[\left(|Z_1|^{\alpha}-\ell_{\alpha}\right)^2\left(|Z_2|^{\alpha}-\ell_{\alpha}\right)^2\right],
\end{align}
where $Z_1,Z_2$ are i.i.d standard Gaussians.

If we prove that 
\begin{align}\label{eq:4V}
\lim_{n\rightarrow\infty}\sum_{i,j,k\neq 1}\mathbb{E}\left[(E_m)_{1i}(E_m)_{ij}(E_m)_{jk}(E_m)_{k1}\right]=0, 
\end{align}
then using \eqref{eq:2V}, \eqref{eq:3V}, \eqref{eq:4V}, we would have proved that fourth moments of $\bar{\mu}_{E_m}$ are uniformly bounded and we would be done with the proof of Lemma \ref{alpha<1}. Note that 
\begin{align}\label{eq:Y}
\lim_{n\rightarrow\infty}\sum_{i,j,k\neq 1}\mathbb{E}\left[(E_m)_{1i}(E_m)_{ij}(E_m)_{jk}(E_m)_{k1}\right]=
\end{align}
\begin{align*}  
\lim_{n\rightarrow\infty} m\mathbb{E}\left[\left(\left|\frac{\langle R_1,R_2\rangle}{\sqrt{n}}\right|^{\alpha}-\ell_{\alpha}\right)\left(\left|\frac{\langle R_2,R_3\rangle}{\sqrt{n}}\right|^{\alpha}-\ell_{\alpha}\right)\left(\left|\frac{\langle R_3,R_4\rangle}{\sqrt{n}}\right|^{\alpha}-\ell_{\alpha}\right)\left(\left|\frac{\langle R_4,R_1\rangle}{\sqrt{n}}\right|^{\alpha}-\ell_{\alpha}\right)\right]
\end{align*}

Let $\mathcal{F}_{1,3}$ denote the sigma algebra generated from the $1$st row and $3$rd row of $X_n$ and \[Y_{1,3}:=\mathbb{E}\left[\left(\left|\frac{\langle R_1,R_2\rangle}{\sqrt{n}}\right|^{\alpha}-\ell_{\alpha}\right)\left(\left|\frac{\langle R_2,R_3\rangle}{\sqrt{n}}\right|^{\alpha}-\ell_{\alpha}\right)\biggr|\mathcal{F}_{1,3}\right].\]

Note that using independence of $2$nd row and $4$th row of $X_n$, RHS of  \eqref{eq:Y} can be written as, 
$\lim_{n\rightarrow\infty}m\mathbb{E}[Y_{1,3}^2]$. 

We prove the below lemma from which it follows that $\lim_{n\rightarrow\infty}m\mathbb{E}[Y_{1,3}^2]=0$ and hence the fourth moments of $\bar{\mu}_{E_m}$ are uniformly bounded. Let $\rho_{ij}:=\frac{\langle R_i,R_j\rangle}{\lVert R_i\rVert\lVert R_j\rVert}$ and $\sigma_i:=\lVert R_i\rVert/\sqrt{n}$.
\begin{lemma}\label{main lemma_1}
$\mathbb{E}[(nY_{1,3})^k]$ is uniformly bounded by $M_k, \forall n,k\in\mathbb{N}$, where $M_k>0$ are some constants dependent on $k$.
\end{lemma}
\begin{proof}
\begin{align*}
Y_{1,3}&=\mathbb{E}\left[\left(\left|\frac{\langle R_1,R_2\rangle}{\sqrt{n}}\right|^{\alpha}-\ell_{\alpha}\right)\left(\left|\frac{\langle R_2,R_3\rangle}{\sqrt{n}}\right|^{\alpha}-\ell_{\alpha}\right)\biggr|\mathcal{F}_{1,3}\right]\\
&=\mathbb{E}\left[\sigma_1^{\alpha}\left(\left(\left|\frac{\langle R_1,R_2\rangle}{\sigma_1\sqrt{n}}\right|^{\alpha}-\ell_{\alpha}\right)+\left(\ell_{\alpha}-\frac{\ell_{\alpha}}{\sigma_1^{\alpha}}\right)\right)\sigma_3^{\alpha}\left(\left(\left|\frac{\langle R_2,R_3\rangle}{\sigma_3\sqrt{n}}\right|^{\alpha}-\ell_{\alpha}\right)+\left(\ell_{\alpha}-\frac{\ell_{\alpha}}{\sigma_3^{\alpha}}\right)\right) \biggr|\mathcal{F}_{1,3}\right]\\
&=\sigma_1^{\alpha}\sigma_3^{\alpha}\mathbb{E}[(|Z_1|^{\alpha}-\ell_{\alpha})(|Z_3|^{\alpha}-\ell_{\alpha})]+\ell_{\alpha}^2(\sigma_1^{\alpha}-1)(\sigma_3^{\alpha}-1).
\end{align*}
Here $Z_1,Z_3$ are standard normal random variables (after conditioning on $\mathcal{F}_{1,3}$) with correlation coefficient $\rho_{13}$. Note that almost surely $0<|\rho_{13}|<1$ and hence $(Z_1,Z_3)$ have joint density.

Define a function of correlation coefficient as below,
\begin{align*}
I(\rho):&=\frac{1}{2\pi\sqrt{1-\rho^2}}\int_{\mathbb{R}}\int_{\mathbb{R}}(|x|^{\alpha}-\ell_{\alpha})(|y|^{\alpha}-\ell_{\alpha})\exp\left(-\frac{1}{2(1-\rho^2)}\left(x^2+y^2-2xy\rho\right)\right)dxdy.
\end{align*}

Note that $I(0)=0,\  I(\rho)=I(-\rho)$ and $I(\rho)$ is a smooth function. Above given expansion of $Y_{1,3}$ can be written as 
\begin{align*}
Y_{1,3}&=\sigma_1^{\alpha}\sigma_3^{\alpha}\rho_{13}^2\frac{I(\rho_{13})}{\rho_{13}^2}+\ell_{\alpha}^2(\sigma_1^{\alpha}-1)(\sigma_3^{\alpha}-1).
\end{align*} 
We now show $I(\rho)/\rho^2$ is a bounded function.
Fix $t>0$. For $|\rho|>t$, note that $I(\rho)$ is Gaussian expectation and therefore $I(\rho)/\rho^2$ is bounded. We use L'Hospital's rule to get a bound on $\frac{I(\rho)}{\rho^2}$ when $|\rho|<t$.  Using differentiation under integral sign, and using L'Hospital's rule twice, it can be seen that $I(\rho)/\rho^2$ is a bounded function. Hence we can write,
\begin{align*}
|Y|\leq M\sigma_1^{\alpha}\sigma_{3}^{\alpha}|\rho_{13}^2|+m_{\alpha}^2|\sigma_1^{\alpha}-1||\sigma_3^{\alpha}-1|.
\end{align*}
As $\forall\alpha<2$,
\begin{align}\label{alpha<2 inequality}
|\sigma_1^{\alpha}-1|\leq \left|\frac{\langle R_1,R_1\rangle}{n}-1\right|\leq \frac{1}{\sqrt{n}}\left|\frac{\langle R_1,R_1\rangle-n}{\sqrt{n}}\right|.
\end{align}
As a result we can write,
\begin{align*}
|nY|\leq M \sigma_1^{\alpha}\sigma_{3}^{\alpha}n{\rho_{13}^2}+\left|\frac{\langle R_1,R_1\rangle-n}{\sqrt{n}}\right|\left|\frac{\langle R_3,R_3\rangle-n}{\sqrt{n}}\right|.
\end{align*}
It is easy to see that, the $k$th moments of $\sigma_1^{\alpha},\sigma_{3}^{\alpha},n{\rho_{13}^2},\left|\frac{\langle R_1,R_1\rangle-n}{\sqrt{n}}\right|$ are uniformly bounded by some constant, $\forall n\in \mathbb{N}$ and hence $k$th moments of $nY$ are also uniformly bounded . This completes the proof of Lemma \ref{main lemma_1}. 
\end{proof}

This proves that the fourth moments are uniformly bounded. This completes the proof of Lemma \ref{Lemma 1 of alpha<1}. 
\end{proof}

\section{$\alpha>s$ range} \label{sec a>1}
In this section we prove Theorem \ref{main theorem} for the range $\alp>s$. We prove the below lemma which implies Theorem \ref{main theorem} for this range of $\alpha$.
\begin{lemma}\label{alpha>1} 
Fix $s<\alpha$ and $0<\varepsilon<1/2$. Then  $\mathbb{P}(\lambda_1(B_{n,\alp,s})>\varepsilon)\rightarrow 1$, as $n\rightarrow \infty$.
\end{lemma}

\begin{proof}[Proof of Lemma ~\ref{alpha>1}]

 For ease of notation, we write $B_{n,\alp,s}$ as $B_m$.
 Define a diagonal matrix $D_m$ such that $D_m(i,i)=B_m(i,i)-\frac{\ell_{\alpha}}{n^{\alpha/2}}$. Let $C_m:=B_m-\left(\frac{\ell_{\alpha}}{n^{\alpha/2}}\right)J_m-D_m$. Note that $C_m(i,j)=\frac{1}{n^{\alpha/2}}\left(\left|\frac{\langle R_i,R_j\rangle}{\sqrt{n}}\right|^{\alpha}-\ell_{\alpha}\right)$ and the diagonal entries of $C_m$ are zero.

We first show that $\mathbb{P}(\lambda_1(C_m)\leq -1+2\varepsilon)\rightarrow 0$. This will complete the proof of Lemma \ref{alpha>1}. This is true as,
using Lemma \ref{Bernstein}, we have
\begin{align}\label{eq:diag}
\mathbb{P}\left(\bigcup_{i=1}^{ m}\left((D_m)_{ii}\leq 1-\varepsilon\right)\right)\leq 2m\exp(-c_3n),
\end{align}
for some constant $c_3>0$ depending on $\alpha$. To get the matrix $B_m$, we add $C_m$ with $D_m+(\ell_{\alpha}/n^{\alpha/2})J_m.$
Using Weyl's inequality (Theorem $4.3.1$ of \cite{HJ}), we get
\begin{align}\label{eq:weyl_2}
\mathbb{P}\left(\lambda_1(B_m)- \lambda_1(C_m)< 1-\varepsilon\right)\leq 2m\exp(-c_3n).
\end{align}

The above inequality shows that the eigenvalues of $B_m$ are at least $1-\varepsilon$ more than that of $C_m$, with high probability. This completes the proof if we prove $\mathbb{P}(\lambda_1(C_m)\leq -1+2\varepsilon)\rightarrow 0$.

Choose $k$ such that $\alpha>\left(\frac{k+1}{k}\right)s$.
\begin{align*}
\mathbb{P}(\lambda_1(C_m)\leq -1+2\varepsilon)&\leq \mathbb{P}((\lambda_1(C_m))^{2k}\geq (1-2\varepsilon)^{2k})\\
&\leq \frac{\mathbb{E}[\mathrm{Tr}(C_m^{2k})]}{(1-2\varepsilon)^{2k}}.
\end{align*}

We prove that $\mathbb{E}[\mathrm{Tr}(C_m^{2k})]\rightarrow 0$, where $\alpha>\left(\frac{k+1}{k}\right)s$. This completes the proof of the theorem.

 


We state a lemma here which generalizes Lemma \ref{main lemma_1}. Let $p\in \mathbb{N}_{\geq 3}.$ 
Define $$G:={n^{(2(p-3)+2)\alpha/2}}\mathbb{E}[C_{12}C_{13}C_{14}^2C_{15}^2\dots C_{1p}^2\bigr|\mathcal{F}_{2,3,\dots,p}]$$
\begin{lemma}\label{main lemma_2}
$\mathbb{E}[(nG)^k]$ is uniformly bounded by constant $M_k$ for all $k\in\mathbb{N}$.
\end{lemma}

\begin{proof}

Let $W_{12}:=\bigg(\bigg|\frac{\langle R_1,R_2\rangle}{\sigma_2\sqrt{n}}\bigg|^{\alpha}-\ell_{\alpha}\bigg)+ \frac{\ell_{\alpha}}{\sigma_2^{\alpha}}\bigg(\sigma_2^{\alpha}-1\bigg)$. Then
\begin{align*}
G=\sigma_2^{\alpha}\sigma_3^{\alpha}\sigma_4^{2\alpha}\sigma_5^{2\alpha}\dots\sigma_p^{2\alpha}\mathbb{E}\bigg[W_{12}W_{13}W_{14}^2W_{15}^2\dots W_{1p}^2\biggr|\mathcal{F}_{2,3,\dots,p}\bigg].
\end{align*}
Due to \eqref{alpha<2 inequality}, the term $(\sigma_2^\alpha-1)$ is of the order of $1/\sqrt{n}$. All moments of $\sigma_2^{\alpha}$ are uniformly bounded. So for $\mathbb{E}[(nG)^k]$ to be uniformly bounded, it is enough to prove that $k$-th moments of 
\begin{align*}
n\mathbb{E}\left[\bigg(\bigg|\frac{\langle R_1,R_2\rangle}{\sigma_2\sqrt{n}}\bigg|^{\alpha}-\ell_{\alpha}\bigg)\bigg(\bigg|\frac{\langle R_1,R_3\rangle}{\sigma_2\sqrt{n}}\bigg|^{\alpha}-\ell_{\alpha}\bigg)W_{14}^2\dots W_{1p}^2\biggr|\mathcal{F}_{2,3,\dots,p}\right]
\end{align*} and
\begin{align*}
\sqrt{n}\mathbb{E}\left[\bigg(\bigg|\frac{\langle R_1,R_2\rangle}{\sigma_2\sqrt{n}}\bigg|^{\alpha}-\ell_{\alpha}\bigg)W_{14}^2\dots W_{1p}^2\biggr|\mathcal{F}_{2,3,\dots,p}\right]
\end{align*}
are uniformly bounded, $\forall n\in\mathbb{N}$. We will prove that $k$-th moment of first quantity is uniformly bounded. For the second quantity, similar argument works.

Note that conditional on $\mathcal{F}_{2,3,\dots,p}$, the conditional expectation $G$ is a function of standard Gaussian random variables, say, $Z_2,Z_3,\dots,Z_p$, with the correlation matrix being $\tilde{\Sigma}=A_{p-1}A_{p-1}^T$, where $A_{p-1}$ is $(p-1)\times n$ matrix with $A_{p-1}(i,j)=\frac{X_n(i+1,j)}{\sqrt{n}\sigma_{i+1}}$. It can be seen easily that almost surely rank$(A_{p-1})=p-1$ and hence $\tilde{\Sigma}$ is invertible. For any symmetric invertible matrix $\Sigma$ with $1$s on diagonal, define
\begin{align*}
h(\Sigma):=\frac{1}{\sqrt{(2\pi)^{p-1}|\Sigma|}}\int(|x_1|^{\alpha}-\ell_{\alpha})(|x_2|^{\alpha}-\ell_{\alpha})\dots (|x_{p-1}|^{\alpha}-\ell_{\alpha})^2\exp\left(\frac{-x^T\Sigma^{-1}x}{2}\right)dx_1\dots dx_{p-1}
\end{align*}
Here $h$ is a function of the entries above the diagonal of $\Sigma$. Using symmetry and independence $h(I_{p-1})=0$. Expanding $W_{12}W_{13}W_{14}^2W_{15}^2\dots W_{1p}^2$ and using the fact that $(\sigma_2^\alpha-1)$ is of order $1/\sqrt{n}$, to prove that $k$-th moments of 
\begin{align*}
    n\mathbb{E}\left[\bigg(\bigg|\frac{\langle R_1,R_2\rangle}{\sigma_2\sqrt{n}}\bigg|^{\alpha}-\ell_{\alpha}\bigg)\bigg(\bigg|\frac{\langle R_1,R_3\rangle}{\sigma_2\sqrt{n}}\bigg|^{\alpha}-\ell_{\alpha}\bigg)W_{14}^2\dots W_{1p}^2\biggr|\mathcal{F}_{2,3,\dots,p}\right]
\end{align*}
are uniformly bounded, it is enough to prove that $k$-th moments of $nh(\tilde{\Sigma})$ are uniformly bounded.
%

It is easy to see that $h$ is a differentiable function. We make use of the multi-variable mean value theorem $|f(y)-f(x)|\leq |\nabla f(cx+(1-c)y)||y-x|,$ for some $ 0\leq c\leq 1$. Using the fact that $\sum_{i< j}\tilde{\Sigma}_{i,j}^2$ is of order of $1/\sqrt{n}$, it is enough to show $h(\Sigma)/\sum_{i< j}\Sigma_{i,j}^2$ is bounded.

 For $\Sigma$ bounded away from the origin, using Gaussian integrals, it can be seen that $\frac{h(\Sigma)}{\sum_{i< j}\Sigma_{i,j}^2}$ is bounded. As $h(I_{p-1})=0$ at the origin, mean value theorem and basic computations gives boundedness of $h(\Sigma)/\sum_{i< j}\Sigma_{i,j}^2$ in a neighbourhood of the origin. This completes the proof of Lemma \ref{main lemma_2}. 
\end{proof} 

\textbf{Computation of $\mathbb{E}[\mathrm{Tr}(C_m^{2k})]$}: Consider a closed walk of length $2k$ on complete graph $K_m$. Let $i_1i_2\dots i_{2k-1}i_1$ be the closed walk. This corresponds to the term $\mathbb{E}[C_{i_1i_2}C_{i_2i_3}\dots C_{i_{2k-1}i_1}]$ in expansion of $\mathbb{E}[\mathrm{Tr}(C_m^{2k})]$.  Thus terms in expansion of  $\mathbb{E}[\mathrm{Tr}(C_m^{2k})]$ correspond to closed walks of length $2k$ (starting point can be any of the $m$ vertices).  As the diagonal entries are zero, the paths cannot have loops at any vertices. We first look at walks without "leaf vertices". By "leaf vertices" we mean the vertices, like "$3$" and "$1$", which are of degree $2$ as shown below (In the graph generated due to closed walk, such vertices are leafs). 
\begin{figure}[h!]
\includegraphics[scale=0.7]{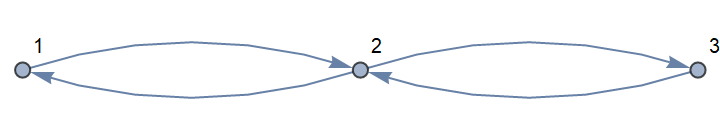}
\centering
\caption{The vertices $1,3$ are leaf vertices}
\end{figure}

So we look at closed walks of length $2k$ without loops and leaf vertices. As the off-diagonal entries of $C_m$ are of the order $1/n^{\alpha/2}$ and $\alpha>\left(\frac{k+1}{k}\right)s$, the sums of expectations corresponding to paths visiting $k+l$ vertices with $l\leq1$(each vertex can be chosen in at most $\lfloor n^s\rfloor$ ways), goes to $0$. So it is enough to look at paths visiting at least $k+2$ vertices.   

Closed walks of length $2k$, visiting $k+l,\ l\geq 2$ vertices, must have at least $2l$ vertices of degree $2$ (none of which are leaf vertices) as shown below. This is due to the fact that since it is a closed walk, degree of every vertex is even and sum of degrees of vertices must equal twice the total number of edges. 

There would be $C_{i,j}C_{j,k}$ term when expanding $\mathrm{Tr}(C_m^{2k})$ as sum of product of entries of $C_m$. This factor shows up due to the vertex $j$ having degree $2$. We would like to condition on the rows $i,k$ of $X_n$ and use Lemma \ref{main lemma_1}.

It could happen that more than $1$, say $t$, degree-$2$ vertices come together in series as shown below. In such a case we condition as shown in the example below.
 
\begin{figure}[h!]
\includegraphics[scale=0.7]{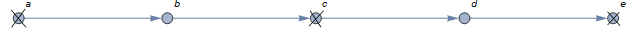}
\centering
\end{figure}

Suppose there is a path traversing vertices $a$ through $e$, as shown above, where degrees of both $a,e$ are at least $4$ and $b,c,d$ are all degree-$2$ vertices. Here degrees are calculated in the graph generated by the closed walk of length $2k$. In such a case we will have the factor $C_{a,b}C_{b,c}C_{c,d}C_{d,e}$ in the expansion of $\mathrm{Tr}(C_m^{2k})$ corresponding to that path. In the expectation term corresponding to such a path, we condition on rows $a,c,e$ and use independence to get $2$ conditional expectations $Y_{a,c},Y_{c,e}$ mentioned in Section \ref{sec a<1} . The `x' mark denotes the rows which we are going to condition on. If there are even number of degree-$2$ vertices coming together, we condition as shown below. 
\begin{figure}[h!]
\includegraphics[scale=0.7]{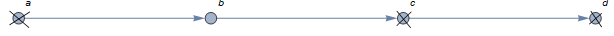}
\centering
\end{figure}

In the case shown above, vertices $a,d$ have degree at least $4$ and $b,c$ are degree-$2$ vertices. We condition of rows $a,c,d$. All other rows corresponding to vertices with degrees greater than $2$ will also be conditioned.

Now we look at $\mathbb{E}[\mathrm{Tr}(C_m^{2k})]$ and the walks of length $2k$, without loops and leaf vertices , visiting $k+l$ vertices. The $k+l$ vertices can be chosen in $\lfloor n^{s}\rfloor^{k+l}$ ways and taking the order of $C_{i,j}$ into account we can write,
\begin{align*}
\frac{\lfloor n^{s}\rfloor^{k+l}}{n^{k\alpha}}\mathbb{E}\left[\left(\left|\frac{\langle R_1,R_2\rangle}{\sqrt{n}}\right|^{\alpha}-\ell_{\alpha}\right)\dots\right]
\end{align*}   
corresponding to the walks we are interested in. Using Independence and conditioning on the rows corresponding to the vertices with degree at least $4$ and those appropriate vertices when more than $1$ degree-$2$ vertices come together, we get product of at least $l$ number of conditional expectations like $Y_{i,j}$. Using Lemma \ref{main lemma_1}, $n^l\mathbb{E}\left[\left(\left|\frac{\langle R_1,R_2\rangle}{\sqrt{n}}\right|^{\alpha}-\ell_{\alpha}\right)\dots\right]$ is uniformly bounded. As $l$ was arbitrary and as $\alpha>\left(\frac{k+1}{k}\right)s$, we can see that the expectation corresponding to the walks without loops and leaf vertices goes to $0$ with $n$.

Now we look at paths without loops but have leaf vertices. If initially we had a closed walk of length $2g$ without leaf vertices and visited $l$ different vertices.   Note that each leaf vertex attached increases length of walk by $2$ and number of vertices visited by $1$. Adding $t$ leaf vertices such that $g+t=k$ gives corresponding expectation terms like
\begin{align*}
\frac{\lfloor n^{s}\rfloor^{g+l+t}}{n^{(g+t)\alpha}}\mathbb{E}\left[\left(\left|\frac{\langle R_1,R_2\rangle}{\sqrt{n}}\right|^{\alpha}-\ell_{\alpha}\right)\dots\right].
\end{align*} 
 If such a leaf vertex or multiple leaf vertices can be attached to a vertex which is degree-$2$ originally, then we condition on the rows corresponding to all the leaf vertices and the vertices whose rows we were conditioning on originally, as shown below. 
\begin{figure}[h]
\includegraphics[scale=0.4]{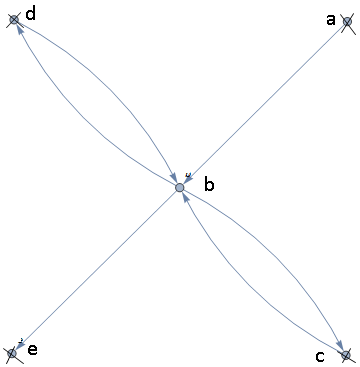}
\centering
\end{figure}
The vertices $d,c$ are leaf vertices attached to vertex $b$.
 Without the vertices $d,c$ and edges between them and $b$, the vertex $b$ would be of degree-$2$. After addition of vertices $d,c$ and the edges, the conditioning will be done on rows corresponding to $a,c,d,e$. This is where Lemma \ref{main lemma_2} is used. Such conditioning gives conditional expectation factor like $G$ in Lemma \ref{main lemma_2} for every vertex which get attached at least one leaf vertex to it. 
 
 If leaf vertices are attached to a vertex which is of degree $4$ or more originally, then again we condition on rows corresponding to all leaf vertices along with the previous vertices we were conditioning on (Lemma \ref{main lemma_2} is not needed here). As $G$ is of the order of $1/n$ and $\alpha>\left(\frac{k+1}{k}\right)s$,
 \begin{align*}
\frac{\lfloor n^s\rfloor^{g+t+l}}{n^{(g+t)\alpha}}\mathbb{E}\left[\left(\left|\frac{\langle R_1,R_2\rangle}{\sqrt{n}}\right|^{\alpha}-\ell_{\alpha}\right)\dots\right]\rightarrow 0
\end{align*}  
This shows that $\mathbb{E}[\mathrm{Tr}(C_n^{2k})] \rightarrow 0$, as $n\rightarrow \infty$. Taking $k$ arbitrarily large completes the proof of Lemma \ref{alpha>1}.
\end{proof}

\subsection*{Acknowledgement.} The author thanks Manjunath Krishnapur for suggesting the question addressed in this article and for several helpful discussions without which this article could not have been possible.
\bibliographystyle{abbrvnat}
\bibliography{biblio_lin_alg}

\end{document}